\documentclass{amsart}

\usepackage{graphicx,fancybox,amssymb,color}

 \usepackage{framed}
\definecolor{shadecolor}{gray}{.9}

 \newcommand{\arxiv}[1]{\begin{leftbar}\begin{shaded}
  #1\end{shaded}\end{leftbar}}

\newtheorem{theorem}{Theorem}[section]
\newtheorem{corollary}[theorem]{Corollary}

\newtheorem{lemma}[theorem]{Lemma}
\newtheorem{proposition}[theorem]{Proposition}

\theoremstyle{definition}
\newtheorem{definition}[theorem]{Definition}
\theoremstyle{remark}
\newtheorem{remark}[theorem]{Remark}
\theoremstyle{remark}
\newtheorem{example}{Example}[section]

\numberwithin{equation}{section}

\newcommand{\RR}{\mathbb{R}}

\newcommand{\Wq}{B^{(q)}}

\newcommand{\dd}{d\!\!l }
\newcommand{\calD}{\mathcal{D}}

\newcommand{\toD}{\xrightarrow{{\mathcal{D}}}}

\title{On integration with respect to the $q$-Brownian motion}
\author{
W{\l}odek  Bryc
}

\address{
Department of Mathematical Sciences,
University of Cincinnati,
PO Box 210025,
Cincinnati, OH 45221--0025, USA}
\email{Wlodzimierz.Bryc@UC.edu}
\date{Created: Dec. 2013. Revised: May 2014 and July 2014. Printed \today. (File: \jobname.tex)}

\begin{document}

\maketitle

\begin{abstract}
  For a parameter  $0<q<1$, we use the Jackson $q$-integral to define integration with respect to the so called $q$-Brownian motion.
 Our main results are the $q$-analogs of the $L_2$-isometry and of the  Ito formula for polynomial integrands.
We also indicate how  the $L_2$-isometry  extends the integral to more general functions.
   % in such a way that as $q\to 1$ we get the stochastic Ito integral with respect to the usual Brownian motion.

%\longcomment{
%  For a parameter  $0<q<1$, we introduce integration with respect to the $q$-Brownian motion
%    in such a way that as $q\to 1$ we get the stochastic Ito integral with respect to the usual Brownian motion. Our main results are the analogs of the $L_2$-isometry and of the  Ito formula.
%}

\end{abstract}

\arxiv{{\bf Note}  This is the expanded version of the manuscript with additional material which  is marked by the gray background like in  this note.

Please note that the numbering of formulas in this version does not match the shorter version!!!  }

\section{Introduction}

In this paper we use the Jackson $q$-integral   to define an operation of integration with respect to a  Markov process which  first arose in non-commutative probability and which we shall call the $q$-Brownian motion.

The   non-commutative $q$-Brownian motion was  introduced in \cite{BS91} who put the formal approach of   \cite{frisch2003parastochastics} on firm mathematical footing.  Stochastic integration with respect to the non-commutative $q$-Brownian was developed in Ref.  \cite{donati2003stochastic}.

According to  \cite[Corollary 4.5]{BKS-97},  there exists a unique classical  Markov process   $(\Wq)_{t\in[0,\infty)}$ with the same univariate distributions and the same transition operators as the non-commutative $q$-Brownian motion.     Markov process $(\Wq)_{t\geq 0}$  is a martingale and which   converges in distribution to the Brownian motion as $q\to 1$.
With some abuse of terminology we call  $(\Wq_t)$  the $q$-Brownian motion and we review its basic properties in Section \ref{Sec-qB}.

Our definition of integration with respect to  $(\Wq_t)$ uses the orthogonal martingale polynomials   and mimics the well-known properties of the Ito integral.
The integral of an instantaneous function $f(\Wq_s,s)$ of the process is denoted by $$\int_0^t f(\Wq_s,s)\dd \Wq,$$
where  to avoid confusion with the Ito integral with respect to the martingale $(\Wq_t)$ we use $\dd$ instead of $d$.

We define this integral for  $0<q<1$ and for polynomials $f(x,t)$ in variable $x$ with coefficients that are bounded for small enough $t$.
Our definition uses the Jackson integral  \cite{Jackson-1910} which is reviewed in Section  \ref{Sec-Jackson}.  This approach naturally  leads to the $q$-analog of the $L_2$-isometry
 \begin{equation}\label{EQ-iso}
   E\left(\left|\int_0^t f(\Wq_s,s)\dd \Wq\right|^2\right)=\int_0^tE\left(\left|f(\Wq_s,s)\right|^2\right) d_qs,
 \end{equation}
 with the Jackson $q$-integral appearing on the right hand side.  In Theorem \ref{Thm1} we establish this formula for  polynomial $f(x,t)$ and   in  Corollary \ref{Cor} we use it to extend the integral to more general functions.   Proposition \ref{P4.5} shows that the integral is well defined for analytic functions $f(x)$  that do not depend on $t$. In Section \ref{Sec-SDE} we use Corollary \ref{Cor}   to exhibit a solution of the "linear $q$-equation" $\dd Z = a Z \dd\Wq$.

 Our second main result is a version of Ito's formula which takes the form
 \begin{multline}\label{EQ-ito}
   f(\Wq_t,t)-f(0,0)\\ =\int_0^t (\nabla_{x}^{(s)} f)(\Wq_s,s)\dd\Wq+\int_0^t (\calD_{q,s} f) (\Wq_{qs},s)d_qs +\int_0^t (\Delta_{x}^{(s)} f) (\Wq_{qs},s)d_qs.
 \end{multline}
 In Theorem \ref{Thm2}, this formula is
  %These formulas will be
  established for polynomial $f(x,t)$ and $q\in(0,1)$.
  In Remark  \ref{Rem-Ito} we point out that we expect this formula to hold in more generality.

  As in  the Ito formula, the operators $\calD_{q,s}$, $\nabla_x^{(s)}$ and $\Delta_x^{(s)}$ should be interpreted as acting on the appropriate variable of the function $f(x,s)$, which is then evaluated at $x=\Wq_s$ or at $x=\Wq_{qs}$.
The  time-variable operator $\calD_{q,s} f$ is the  $q$-derivative with respect to variable $s$:
\begin{equation}\label{Def:Dq}
 (\calD_{q,s} f)(x,s)=\frac{f(x,s)-f(x,qs)}{(1-q)s}.
\end{equation}
This expression is well defined for  $s>0$ which is all that is needed in  \eqref{EQ-ito} when $q\in(0,1)$.

 The  operators $\nabla_x^{(s)}$ and $\Delta_x^{(s)}$   act only on the ``space variable" $x$ but they depend on the time variable $s$, and are defined as the ``singular integrals" with respect to two time-dependent probability kernels:
 \begin{eqnarray}
   \label{EQ:nabla} \\ \nonumber
   (\nabla_{x}^{(s)} f)(x,s)&=&\int_\RR \frac{f(y,s)-f(x,s)}{y-x}\nu_{x,s}(dy) \\
   \label{EQ:Delta} \\ \nonumber
   (\Delta_{x}^{(s)} f)(x,s)&=&
   \iint_{\RR^2} \frac{(y-x)f(z,s)+(x-z)f(y,s)+(z-y)f(x,s)}{(x-y)(y-z)(z-x)}\mu_{x,s}(dy,dz).
   %%\iint_{\RR^2} \frac{(z-x)f(y,s)+(x-y)f(z,s)+(y-z)f(x,s)}{(x-y)(x-z)(y-z)}\mu_{x,s}(dy,dz).
 \end{eqnarray}
% \comment{more elegant?}
%$$\frac{ (x-z)y^n+(y-x)z^n + (z-y)x^n}{(z-x)(x-y)  (y-z)}$$
%$$\iint_{\RR^2} \frac{(y-x)f(z,s)+(x-z)f(y,s)+(z-y)f(x,s)}{(x-y))(y-z)(z-x)}\mu_{x,s}(dy,dz)$$
Probability measures $\nu_{x,s}(dy)$  and $\mu_{x,s}(dy,dz)$  can be expressed in terms of the transition probabilities $P_{s,t}(x,dy)$  of the Markov process $(\Wq_t)_{t\in[0,\infty)}$ as follows:
    $$\nu_{x,s}(dy)= P_{sq^2,s}(qx,dy)$$ and  $$\mu_{x,s}(dy,dz)=P_{qs,s}(x,dy)\nu_{y,s}(dz)=P_{qs,s}(x,dy)P_{q^2s,s}(qy,dz).$$ Transition probabilities $P_{s,t}(x,dy)$ appear in \eqref{EQ:Pst} below. The same probability kernel $\nu_{x,s}(dy)$ appears also in \cite[Proposition 21]{anshelevich2013generators}.

As $q\to 1$ probability measures $ \nu_{x,s}(dy)$  and $\mu_{x,s}(dy,dz)$  converge to degenerated measures. Using Taylor expansion  one can check  that for smooth enough functions $ (\nabla_{x}^{(s)} f)(x,s)$ converges to $\partial f/\partial x$ and
 $ (\Delta_{x}^{(s)} f)(x,s)$  converges to  $\tfrac12\partial^2 f/\partial x^2$ as $q\to1$., so \eqref{EQ-ito} is a $q$-analog of the Ito formula.
We note that there seems to be some interest in $q$-analogs of the Ito formula.  In particular,  Ref. \cite{haven2009quantum,haven2011ito} discuss formal $q$-versions of the Ito formula and its
     applications to financial mathematics.

% The paper is organized as follows.  Section \ref{Sec-qB} is a short overview of the $q$-Brownian motion. Section  \ref{Sec-Jackson} is a review of the  Jackson $q$-integral. In Section \ref{Sec-SI} we introduce our integration procedure, and we prove formulas  \eqref{EQ-iso} and \eqref{EQ-ito} for polynomials.
%  In Section \ref{Sec-SDE} we discuss the simplest nontrivial equation $\dd Z = a Z \dd \Wq$. In Section \ref{Sec-Determ} we point out that the distribution of the  integrals $\int_0^t s^r \dd\Wq$ depends in a complicated way on the parameters $r,q$.

\arxiv{
\subsection{Limits as $q\to1$}
We conclude this introduction with the verification that
 as $q\to 1$,   formulas \eqref{EQ-iso} and \eqref{EQ-ito} coincide with the usual formulas of Ito calculus, at least on polynomials.

 Formula \eqref{EQ-iso} becomes the familiar $$E\left(\int_0^t f(W_s,s)dW\right)^2=\int_0^t E(f(W_s,s))^2ds,$$ as
 the Jackson  $q$-integral becomes the Riemann integral in the limit.  As $q\to 1$ formula \eqref{EQ-ito} becomes
  $$
  f(W_t,t)=\int_0^t \left(\frac{\partial}{\partial x}f\right)(W_s,s)dW+\int_0^t \left(\frac{\partial}{\partial s}f\right)(W_s,s)ds+\frac12 \int_0^t \left(\frac{\partial^2}{\partial x^2}f\right)(W_s,s)ds,
  $$
 but this is  less obvious. Clearly, $\calD_{q,s} f$ becomes $\frac{\partial}{\partial s}$ in the limit.
  Since $P_{q^2s,s}(qx,dy)\toD \delta_x(dy)$ in distribution as $q\to 1$ and $\frac{f(y)-f(x)}{y-x}=f'(x)+\frac12 f''(x)(y-x)+\dots$, we see that $\nabla_{x}^{(s)}$ becomes $\frac{\partial}{\partial x}$,   when acting on smooth enough functions of variable $x$. Next, we expand the numerator of the integrand  in \eqref{EQ:nabla} into the Taylor series at $x$ (here we take   $f$ to be a polynomial in variable $x$). We get
 \begin{multline}\label{1.6}
   (y-x)f(z)+ (x-z)f(y)+(z-y)f(x) \\=(y-x)\left[f(x)+f'(x)(z-x)+\frac12 f''(x)(z-x)^2\right]
   \\+(x-z)\left[f(x)+f'(x)(y-x)+\frac12 f''(x)(y-x)^2\right]+(z-y)f(x)\\ +R(x,y,x)
   = \frac12(x-y)(y-z)(z-x)f''(x)+R(x,y,z).
 \end{multline}
 }\arxiv{
After inserting  \eqref{1.6} into \eqref{EQ:Delta}, the first term gives $\tfrac12\tfrac{\partial ^2}{\partial x^2}$, since  $\mu_{x,s}(dy,dz)$ is a probability measure.  %$\mu_{x,s}(dy,dz)\toD \delta_x(dy)\delta_x(dz)$   weakly.
The reminder term $R(x,y,z)$ consists of the higher order terms. The coefficient at  $f^{(k)}(x)/k!$ in $R$ is
 \begin{multline*}
   (y-x)(z-x)^k+(x-z)(y-x)^k=
 (y-x)(z-x)\left((z-x)^{k-1} - (y-x)^{k-1}\right)\\=
 (x-y)(y-z)(z-x) \sum_{j=0}^{k-2} (z-x)^j(y-x)^{k-2-j},
 \end{multline*} and it is clear that for $k\geq 3$ we have $\lim_{q\to 1^-}\int (z-x)^j(y-x)^{k-2-j} \mu_{x,s}(dy,dz)= 0$, since  $\mu_{x,s}(dy,dz)\toD \delta_x(dy)\delta_x(dz)$   in distribution as $q\to1$.
 %(The fact that in the limit as $q\to1$ the integral $\int f \dd \Wq$ becomes the Ito integral $\int f dW$   for a polynomial $f(x,t)$  will be a consequence of our definition.)
 }

\section{$q$-Brownian motion}\label{Sec-qB}

 With $\Wq_0=0$,   the univariate distribution of $\Wq_t$ is a $q$-Gaussian distribution
\begin{multline}
\label{EQ:q-gauss}   \gamma_{t;q}(dy) \\= \frac{\sqrt{1-q}}{2\pi \sqrt{4t-(1-q)y^{2}}}\prod_{k=0}^{\infty
}\left( (1+q^{k})^{2}-(1-q)\frac{y^{2}}{t}q^{k}\right) \prod_{k=0}^{\infty
}(1-q^{k+1}) dy,
\end{multline}
supported on the interval $|y|\leq 2\sqrt{t}/\sqrt{1-q}$.   Here, parameter $t$ is the variance.  (We note that there are   several other distributions that are also called  $q$-Gaussian, see the introduction to \cite{diaz2009gaussian}.)

The transition probabilities $P_{s,t}(x,dy)$ are non-homogeneous and, as noted in \cite{Bryc-Wesolowski-2013-gener}, are well defined for all $x\in\RR$.
The explicit form will not be needed in this paper, but for completeness we note that for  $|x|\leq 2\sqrt{s}/\sqrt{1-q}$ and $s<t$ transition probability  $P_{s,t}(x,dy)$  has  density supported on the interval $|y|\leq 2\sqrt{t}/\sqrt{1-q}$ which can be written as an infinite product
\begin{multline}
  \label{EQ:Pst}
  P_{s,t}(x,dy)\\ = \frac{\sqrt{1-q}}{2\pi \sqrt{4t-(1-q)y^{2}}} \prod_{k=0}^{\infty }\frac{
(t-sq^{k})\left( 1-q^{k+1}\right) \left(
t(1+q^{k})^{2}-(1-q)y^{2}q^{k}\right) }{(t-sq^{2k})^{2}-(1-q)
q^{k}(t+sq^{2k})xy +(1-q)(sy^{2}+tx^{2})q^{2k}}dy.
\end{multline}
With $\Wq_0=0$, we have $\gamma_{t;q}(dy)=P_{0,t}(0,dy)$ and the separable version of the  process  satisfies  $|\Wq_t|\leq 2\sqrt{t}/\sqrt{1-q}$ for all $t\geq 0$ almost surely.
Formulas \eqref{EQ:q-gauss} and \eqref{EQ:Pst} are taken from \cite[Section 4.1]{Bryc-Wesolowski-03}, but they are well known, see  \cite[Theorem 4.6]{BKS-97} and
\cite{anshelevich2013generators,Bryc-Matysiak-Szablowski,szabl2012q}. We note that for $|x|>2\sqrt{s}/\sqrt{1-q}$ transition probability  $P_{s,t}(x,dy)$ has
an additional discrete component; for $q=0$ this discrete component is explicitly written out in \cite[Section 5.3]{Biane98}.

As $q\to1$, the transition probabilities $P_{s,t}(x;dy)$ converge in variation norm to the (Gaussian)  transition probabilities of the Wiener process. This can be deduced from the convergence of densities established in the appendix of \cite{ismail1988askey}.

 Our definition of the integral relies on the so called continuous $q$-Hermite polynomials $\{h_n(x;t):n=0,1,\dots\}$. These are monic polynomials  in variable $x$ that are defined by the three step recurrence
 \begin{equation}   \label{EQ:rec-q-hermite}
 xh_n(x;t)=h_{n+1}(x;t)+t [n] h_{n-1}(x;t),
 \end{equation}
 where $[n]=[n]_q=1+q+\dots+q^{n-1}$, $h_0(x;t)=1$, $h_1(x;t)=x$. (These are renormalized monic versions of the ``standard" continuous $q$-Hermite polynomials as given in \cite[Section 13.1]{Ismail-05} or in \cite[Section 3.26]{Koekoek-Swarttouw}.)

 Polynomials $\{h_n(x;t)\}$ play a special role because they  are orthogonal martingale polynomials for $(\Wq)$. The orthogonality is
 \begin{equation}\label{EQ:ortho}
   \int h_n(x;t)h_m(x;t)\gamma_{t;q}(dx)=\delta_{m,n} [n]! t^n,
 \end{equation}
where $[n]!=[n]_q!=[1]_q[2]_q\dots[n]_q$. The martingale property
\begin{equation}\label{EQ:mart}
h_n(x;s)=\int h_n(y;t)P_{s,t}(x,dy)
\end{equation}
holds for all real $x$ and $s<t$.
 When $q=1$ recursion \eqref{EQ:rec-q-hermite} becomes the recursion for  the Hermite polynomials,  which are the martingale orthogonal polynomials for   the Brownian motion \cite{Schoutens00}.

It is easy to check from \eqref{EQ:rec-q-hermite} that
\begin{equation}
  \label{EQ:h(t)2h(1)}
  h_n(x;t)=t^{n/2} h_n(x/\sqrt{t};1).
\end{equation}
From \eqref{EQ:h(t)2h(1)} it is clear that if $|x|\leq 2\sqrt{t}/\sqrt{1-q}$ then
\begin{equation}\label{EQ:h-growth}
   |h_n(x;t)|\leq C_n t^{n/2}
\end{equation}
for some constant $C_n$.
Explicit sharp bound with  $C_n=(1-q)^{-n/2}\sum_{k=0}^n[^n_k]$ can be deduced from \cite[(13.1.10)]{Ismail-05}.

%It is known that time-inversion $\widetilde\Wq_t=t\Wq_{1/t}$ is a $q$-Brownian motion.
\section{The Jackson $q$-integral}\label{Sec-Jackson}
For reader's convenience we review basic facts about the Jackson integral \cite{Jackson-1910}. In addition to introducing the notation,   we explicitly spell out the ``technical" assumptions that suffice for our purposes. In particular, since we will integrate functions defined on  $[0,\infty)$ only, we  assume that  $q\in(0,1)$. We follow  \cite{kac2002quantum} quite closely, but essentially the same material can be found in numerous other sources such as  \cite[Section 11.4]{Ismail-05}.

Suppose  $a:[0,\infty)\to \RR$ is bounded in a neighborhood of $0$, and let $b:[0,\infty)\to\RR$ be such that
\begin{equation}
  \label{EQ:bound}
  |b(t)-b(0)|\leq C t^\delta
\end{equation} for some $C<\infty$ and $\delta>0$ in a neighborhood of $0$.

\begin{definition}  For $t>0$ and $q\in(0,1)$, the Jackson $q$-integral of $a(s)$ with respect to $b(s)$ is defined as
\begin{equation}
  \label{J-int}
  \int_0^t a(s)d_q b(s):=\sum_{k=0}^\infty a(q^k t)\left(b(q^k t)-b(q^{k+1}t)\right).
\end{equation}
\end{definition}
 When $b(t)=t$, formula \eqref{J-int} takes the following form that goes back to \cite{Jackson-1910}
 \begin{equation}
   \label{J-ds-int}
  \int_0^t a(s)d_q s=(1-q)t\sum_{k=0}^\infty q^k a(q^k t).
 \end{equation}
 Recalling \eqref{Def:Dq}, it is clear that \eqref{J-int} is $\int_0^t a(s) (\calD_{q,s} b)(s) d_q s$.

We will also need the $q$-integration by parts formula \cite[\S5]{Jackson-1910}. This formula is derived in \cite[page 83]{kac2002quantum} under the assumption of differentiability of functions $a(t), b(t)$. Ref. \cite[Theorem 11.4.1]{Ismail-05} gives a $q$-integration by parts formula under more general assumptions, but since the conclusion is stated in a different form than what we need,  we give  a short proof.
\begin{proposition}
 For any pair of functions $a,b$ that satisfy bound \eqref{EQ:bound} near $t=0$ we have
\begin{equation}\label{EQ:by-parts}
 \int_0^t a(s)d_q b(s)= a(t)b(t)-a(0)b(0)-\int_0^t b(q s) d_q a(s).
\end{equation}
 \end{proposition}
 \begin{proof}
Under our assumptions on $a,b$, the series  defining $\int_0^t a(s)d_q b(s)$ and $\int_0^t b(q s) d_q a(s)$ converge, so  we can pass to the limit  as $N\to\infty$  in the telescoping sum identity
\begin{multline*}
\sum_{n=0}^N a(q^n t)\left(b(q^n t)-b(q^{n+1}t)\right)+\sum_{n=0}^N b(q^{n+1}t)\left(a(q^n t)-a(q^{n+1}t)\right) \\ =a(t)b(t)-a(q^{N+1}t)b(q^{N+1}t).
\end{multline*}

 \end{proof}
We also note the $q$-anti-differentiation formula, which is a special case $a(s)=1$ of formula \eqref{EQ:by-parts}
\begin{equation}
  \label{EQ:antideriv}
  b(t)-b(0)=\int_0^t (\calD_{q,s} b)(s)d_qs.
\end{equation}

\section{Integration of polynomials with respect to $\Wq_t$}\label{Sec-SI}
Our definition is designed to imply the relation
\begin{equation}\label{WdW}
 \int_0^t h_n(\Wq_s,s)\dd\Wq=    \frac{1}{[n+1]_q}h_{n+1}(\Wq_t;t)
\end{equation}
and relies on expansion into polynomials $\{h_n(x;t)\}$.

Let $q\in(0,1)$ and suppose that $f(x,t)=\sum_{m=0}^d a_m(t)x^m$ is a polynomial of degree $d$ in variable $x$ with coefficients $a_m(t)$ that depend on $t\geq 0$.
Then, for $t>0$, we can expand $f$ into the $q$-Hermite polynomials \eqref{EQ:rec-q-hermite},
\begin{equation}\label{EQ:f2h}
f(x,t)=\sum_{m=0}^d \frac{b_m(t)}{[m]!}h_m(x;t).
\end{equation}
%\longcomment{
%If is clear that if for each $x$ the function $f(x,t)$ is bounded as a function of $t$ in a neighborhood of $t=0$ then the coefficients $a_k(t)$ are bounded. Similarly, if there exists $\delta>0$ such that $|f(x,t)-f(x,0)|\leq C(x)t^\delta$ in some neighborhood of $0$ then the coefficients $a_k(t)$  satisfy estimate \eqref{EQ:bound}.}
\begin{lemma}\label{Lem:1}  If  all the coefficients $a_m(t)$ of the polynomial $f(x,t)$ are bounded in a neighborhood of $0$ then the coefficients $b_m(t)$ in expansion \eqref{EQ:f2h} are bounded in the same neighborhood of $0$. If all the coefficients $a_m(t)$  satisfy estimate \eqref{EQ:bound} then
  coefficients $b_m(t)$ satisfy estimate \eqref{EQ:bound}.
\end{lemma}\begin{proof}
 Since $h_0,\dots,h_{m-1}$ are monic, they span the polynomials of degree ${m-1}$ so that by orthogonality \eqref{EQ:ortho} we have $\int x^j h_m(x,t)\gamma_{t;q}(dx)=0$ for $j<m$. For the same reason, $\int x^m h_m(x,t)\gamma_{t;q}(dx) =\int h_m^2(x,t)\gamma_{t;q}(dx)=t^m[m]!$. Thus   the coefficients are
\begin{multline*}
  b_m(t)=\frac{1}{t^{m}} \int f(x,t) h_m(x;t)\gamma_{t;q}(dx) =\frac{1}{t^{m}}\sum_{j=m}^d a_j(t) \int x^j h_m(x;t)\gamma_{t;q}(dx)
  \\= \frac{1}{t^{m/2}}\sum_{j=m}^d a_j(t) \int x^j h_m(x/\sqrt{t};1)\gamma_{t;q}(dx),
  \end{multline*}
 where in the last line we used \eqref{EQ:h(t)2h(1)}.  Since random variable $\Wq_t/\sqrt{t}$ has the same distribution as $ \Wq_1$, changing the variable of integration to $y=x/\sqrt{t}$ we get
  \begin{multline*}
  b_m(t)  =\sum_{j=m}^d a_j(t) t^{(j-m)/2}\int y^j h_m(y;1)\gamma_{1;q}(dy) \\=  [m]! a_m(t)+\sum_{j=m+1}^d a_j(t) t^{(j-m)/2}\int y^j h_m(y;1)\gamma_{1;q}(dy).
\end{multline*}
Thus $b_m(0)=[m!]a_m(0)$   and   for $0<t<1$ we get $|b_m(t)-b_m(0)|\leq [m]! |a_m(t)-a_m(0)|+C \sqrt{t}$.
\end{proof}

\begin{definition}
 Let $f(x,t)$ be a polynomial in variable $x$ with coefficients bounded in a neighborhood of $t=0$.
  We define the $\dd \Wq$ integral as the sum of the Jackson $q$-integrals with respect to random functions $h_{m+1}(\Wq_{s};s)$ as follows:
  \begin{equation}
  \label{EQ:q-Int-1}
  \int_0^t f(\Wq_s,s)\dd\Wq= \sum_{m=0}^d  \frac{1}{[m+1]!}\int_0^t b_m(s) d_q h_{m+1}(\Wq_{s};s).
\end{equation}
\end{definition}
 Since $|\Wq_t|\leq 2\sqrt{t}/\sqrt{1-q}$ for all $t\geq 0$ almost surely, from Lemma \ref{Lem:1} and inequality \eqref{EQ:h-growth}
 we see that the Jackson $q$-integrals on the right hand side of \eqref{EQ:q-Int-1} are well defined.
 %\comment{need to extend...}

As a special case we get the following formula for the integrals of the deterministic functions.
\begin{example}\label{Ex:4.1} For non-random integrands, we have
  $$\int_0^t f(s)\dd\Wq=\int_0^t f(s)d_q\Wq_s.$$
 Indeed, since $h_0=1$  the expansion \eqref{EQ:f2h} is $f(s)h_0(x;s)$ and since $h_1(x;s)=x$, we get $ d_q h_{1}(\Wq_{s};s) = d_q\Wq_s$.

 By computing the fourth moments one can check that the distribution of random variable $\int_0^t f(s)d_q\Wq_s$  in general is not $\gamma_{\sigma^2,q}$.
 \end{example}

 When the coefficients of $f(x,t)$  satisfy estimate \eqref{EQ:bound},
we can use \eqref{EQ:by-parts} to write the $\dd \Wq$ integral  as
  \begin{multline}
  \label{EQ:q-Int-2}
  \int_0^t f(\Wq_s,s)\dd\Wq\\=\sum_{m=0}^d   \frac{b_m(t)}{[m+1]!}h_{m+1}(\Wq_t;t)-\sum_{m=0}^d  \frac{1}{[m+1]!}\int_0^t h_{m+1}(\Wq_{qs};qs)d_q b_m(s).
\end{multline}
Taking $b_m(t)=0$ or $1$ we get   formula \eqref{WdW}.

 Formula \eqref{EQ:q-Int-2} leads to a couple of explicit examples.
 \begin{example} \label{Ex:4.2}
 Noting that $x=h_1$,  $x^2=h_2(x;t)+t$, and $h_3(x;t)=x^3-(2+q)xt$, we get the following $q$-analogs of the formulas that are well known for the Ito integrals.
   $$
   \int_0^t \Wq_s\dd\Wq =\frac{1}{1+q} \left((\Wq_t)^2-t\right),
   $$
    $$
   \int_0^t (\Wq_s)^2\dd\Wq =\frac{1}{(1+q)(1+q+q^2)}(\Wq_t)^3-\frac{2+q}{(1+q)(1+q+q^2)}t \Wq_t + \int_0^t s d_q\Wq_s.
   $$
 \end{example}

 We remark that a similar definition could be introduced for a wider class of the $q$-Meixner processes from \cite{Bryc-Wesolowski-2013-gener,Bryc-Wesolowski-03}, compare  \cite[Theorem 7]{Schoutens:998lr}.
\subsection{The $L_2$-isometry}
Our first main result is the $q$-analog of the Ito isometry.
\begin{theorem}\label{Thm1}
  If $0<q<1$ and $f(x,t)$ is a polynomial in $x$ with coefficients bounded in a neighborhood of $t=0$ then \eqref{EQ-iso} holds.
\end{theorem}
\begin{proof}
With \eqref{EQ:f2h}, by orthogonality \eqref{EQ:ortho} for the $q$-Hermite polynomials, we have
$$
  E \left(\left(f(\Wq_s,s)\right)^2\right)=\sum_{m=0}^d \frac{b_m^2(s) s^m}{[m]!}.
$$
So the right hand side of \eqref{EQ-iso} is
\begin{equation}
  \label{EQ:RHS}
  \sum_{m=0}^d \frac{1}{[m]!} \int_0^tb_m^2(s) s^m d_q(s)=\sum_{m=0}^d \frac{1}{[m+1]!} \int_0^tb_m^2(s)  d_q(s^{m+1}).
\end{equation}
On the other hand, expanding the right hand side of \eqref{EQ:q-Int-1} we get
   \begin{multline*}
   E\left(\left(\int_0^t f(\Wq_s,s)\dd \Wq\right)^2\right)
   =\sum_{m,n=0}^d
   \frac{1}{[m+1]![n+1]!}\sum_{k,j=0}^\infty b_m(q^k t)b_n(q^jt) \\
  \times E\left(\left(h_{n+1}(\Wq_{tq^k};tq^k)-h_{n+1}(\Wq_{tq^{k+1}};tq^{k+1}) \right) \left(h_{m+1}(\Wq_{tq^j};tq^j)-h_{m+1}(\Wq_{tq^{j+1}};tq^{j+1})\right)\right).
   \end{multline*}
  Since $h_n(\Wq_t;t)$ is a martingale in the natural filtration, % and $\Wq$ is a Markov process,
  for $j<k$ we have
  $$
  E\left(h_{m+1}(\Wq_{tq^j};tq^j)-h_{m+1}(\Wq_{tq^{j+1}};tq^{j+1})|\Wq_{tq^k}, \Wq_{tq^{k+1}} \right)=0.
  $$
  Similarly,
   $$E\left(h_{n+1}(\Wq_{tq^k};tq^k)-h_{n+1}(\Wq_{tq^{k+1}};tq^{k+1})|\Wq_{tq^j}, \Wq_{tq^{j+1}} \right)=0$$
   for $j>k$. So the only contributing terms are $j=k$ and
    \begin{multline*}
   E\left(\left(\int_0^t f(\Wq_s,s)\dd \Wq\right)^2\right)
   =\sum_{m,n=0}^d
   \frac{1}{[m+1]![n+1]!}\sum_{k=0}^\infty b_m(q^k t)b_n(q^kt) \\ \times
   E\left(\left(h_{n+1}(\Wq_{tq^k};tq^k)-h_{n+1}(\Wq_{tq^{k+1}};tq^{k+1}) \right) \left(h_{m+1}(\Wq_{tq^k};tq^k)-h_{m+1}(\Wq_{tq^{k+1}};tq^{k+1})\right)\right).
   \end{multline*}
   Noting that
   $$ E\left(\left(h_{n+1}(\Wq_{tq^k};tq^k)-h_{n+1}(\Wq_{tq^{k+1}};tq^{k+1}) \right) \left(h_{m+1}(\Wq_{tq^k};tq^k)-h_{m+1}(\Wq_{tq^{k+1}};tq^{k+1})\right)\right)=0$$ for $m\ne n$, we see that
   \begin{multline*}
   E\left(\left(\int_0^t f(\Wq_s,s)\dd \Wq\right)^2\right)\\
   =\sum_{m=0}^d
   \frac{1}{[m+1]!^2}\sum_{k=0}^\infty b_n^2(q^k t)
   E\left(\left(h_{m+1}(\Wq_{tq^k};tq^k)-h_{m+1}(\Wq_{tq^{k+1}};tq^{k+1})\right)^2 \right)
   \\=
  \sum_{m=0}^d
   \frac{1}{[m+1]!}\sum_{k=0}^\infty b_n^2(q^k t)
    t^{m+1}\left(q^{k(m+1)}-  q^{(k+1)(m+1)}\right).
   \end{multline*}
   Here we used the fact that for $s<t$,   by another application of the martingale property and \eqref{EQ:ortho}, we have
   $$E\left(\left(h_k(\Wq_{t};t)-h_k(\Wq_{s};s)\right)^2\right)=E \left(h_k^2(\Wq_{t};t)\right)-E\left(h_k^2(\Wq_{s};s)\right)=[k]!(t^k-s^k).$$

   This shows that the right hand side of \eqref{EQ:q-Int-1} matches \eqref{EQ:RHS}.
\end{proof}
  Suppose that for every $t>0$ the function $f(x,t)$ as a function of $x$ is square integrable  with respect to $\gamma_{t;q}(dx)$ so that we have an $L_2(\gamma_{t;q}(dx))$-convergent expansion
  \begin{equation}\label{EQ-L2}
    f(x,t)=\sum_{n=0}^\infty \frac{b_n(t)}{[n]!}h_n(x;t).
  \end{equation}

\begin{corollary}\label{Cor} If each of the coefficients $b_n(t)$ in \eqref{EQ-L2} is bounded in a neighborhood of $0$ and the series
$$\sum_{n=0}^\infty \frac{1}{[n]!}\int_0^t b_n^2(s)s^n d_qs$$ converges,   then the sequence of random variables
\begin{equation}
  \label{EQ:L2-def}
  \int_0^t \sum_{k=0}^n \frac{b_k(s)}{[k]!}h_k(\Wq_s;s)\dd \Wq
\end{equation}
converges in  mean square as $n\to\infty$.
\end{corollary}
When the mean-square limit \eqref{EQ:L2-def} exists, it is natural to denote it as
\begin{equation}\label{EQ:l.i.m}
  \int_0^t f(\Wq_s,s)\dd\Wq . %=\lim_{n\to \infty} \int_0^t \sum_{k=1}^n \frac{b_n(s)}{[n]!}h_n(\Wq_s;s)\dd \Wq .
\end{equation}

We remark that the verification of the assumptions of Corollary \ref{Cor}  may not be easy. In Section    \ref{Sec-SDE} we consider  function
$$f(x,t)=\prod_{k=0}^\infty \left(1-(1-q) q^k x +t q^{2k}\right)^{-1}$$
and we rely on explicit expansion  \eqref{EQ-L2}  to show that the integral  \eqref{EQ:l.i.m} is well defined for $t<1/(1-q)$. The following result deals with analytic functions that do not depend on variable $t$.
\begin{proposition}\label{P4.5}
Suppose $f(x)=\sum_{k=0}^\infty a_k x^k$ with infinite radius of convergence.   Then $ \int_0^t f(\Wq_s)\dd\Wq$ is well defined for all $t>0$.
\end{proposition}
\begin{proof}
Since  the support of $\gamma_{t;q}(dx)$ is compact  and $f$ is bounded on compacts, it is clear that  we can expand $f(x)$ into  \eqref{EQ-L2}.
 We will show that the coefficients $b_n(t)$ of the expansion are bounded in a neighborhood of $0$. (The proof shows also that $b_n(t)$ satisfy estimate \eqref{EQ:bound};
 the later is not needed here, but this  property is assumed in Theorem \ref{Thm2} below.)

As in the proof of Lemma \ref{Lem:1}, we have
\begin{multline*}
b_n(t)=t^{-n}\int f(x) h_n(x;t)\gamma_{t;q}(dx)
= t^{-n/2}  \int   h_n(y;1) \sum_{k=0}^\infty a_k t^{k/2} y^k\gamma_{1;q}(dy).
\end{multline*}
Since on the support of $\gamma_{1;q}(dy)$ the series in the integrand converges absolutely and is dominated by a constant that does not depend on $y$,  we can switch the order of summation and integration so that
\begin{multline*}
b_n(t) = t^{-n/2}  \sum_{k=0}^\infty a_k t^{k/2}   \int   h_n(y;1)y^k\gamma_{1;q}(dy)\\
= t^{-n/2}  \sum_{k=n}^\infty a_k t^{k/2}   \int   h_n(y;1)y^k\gamma_{1;q}(dy),
\end{multline*}
where in the last line we used   the orthogonality of $h_n(y;1)$ and $y^k$ for $k<n$.

 Choose $R>0$   such that $R^2>4t/(1-q)$. Expressing the coefficients $a_k$ by the contour integrals we get
\begin{multline*}
b_n(t) = \frac{1}{2\pi i}  t^{-n/2}  \sum_{k=n}^\infty \oint _{|z|=R}\frac{f(z)}{z^{k+1}  }dz \, t^{k/2}   \int   h_n(y;1)y^k\gamma_{1;q}(dy)
\\ =\frac{1}{2\pi i} \oint _{|z|=R} f(z)\sum_{k=n}^\infty \frac{t^{k/2-n/2}}{z^{k+1} } t^{k/2}   \int   h_n(y;1)y^k\gamma_{1;q}(dy) dz.
\end{multline*}

Since $|y|\leq 2/\sqrt{1-q}$ on the support of   $\gamma_{1;q}(dy)$, from  \eqref{EQ:ortho} we get
$$
\left| \int   h_n(y;1)y^k\gamma_{1;q}(dy)\right|\leq \left(\frac2{\sqrt{1-q}}\right)^k  \int   |h_n(y;1)|\gamma_{1;q}(dy)\leq \left(\frac2{\sqrt{1-q}}\right)^k  \sqrt{[n]!}.
$$
Summing the resulting geometric series under the $dz$-integral, we get
\begin{equation}
  \label{b-bound}
|b_n(t)|\leq \max_{|z|=R}|f(z)| \sqrt{[n]!}\left(\frac{2}{R\sqrt{1-q}}\right)^n \frac{R\sqrt{1-q}}{R\sqrt{1-q}-2\sqrt{t}}.
\end{equation}
This shows that $b_n(t)$ is bounded in any neighborhood of $t=0$.

Next, we use \eqref{J-ds-int} and inequality \eqref{b-bound} to prove the convergence of the series.
\begin{multline*}
\sum_{n=0}^\infty \frac{1}{[n]!}\int_0^t b_n^2(s)s^n d_qs
=\sum_{n=0}^\infty \frac{(1-q)t^{n+1}}{[n]!}\sum_{k=0}^n b_n^2(q^{k} t) q^{(n+1)k}
\\
\leq C  \sum_{n=0}^\infty \left(\frac{4t}{R^2(1-q)}\right)^n (1-q)\sum_{k=0}^n  q^{k(n+1)}
= C  \sum_{n=0}^\infty \left(\frac{4t}{R^2(1-q)}\right)^n  \frac{1-q}{1-q^{n+1}}  \\ \leq  C  \sum_{n=0}^\infty \left(\frac{4t}{R^2(1-q)}\right)^n <\infty,
\end{multline*}
with constant $C=tR^2(1-q)\max_{|z|=R}|f(z)|^2/(R\sqrt{1-q}-2\sqrt{t})^2 $.

We  now apply Corollary \ref{Cor} to infer that $ \int_0^t f(\Wq_s)\dd\Wq$ is well defined. Since $R$ was arbitrary, the conclusion holds for all $t>0$.
\end{proof}

\subsection{The $q$-Ito formula}
Our second main result is the $q$-version of the Ito formula.
\begin{theorem}\label{Thm2}
  Let $0<q<1$. Suppose $f(x,t)$ is a polynomial in $x$ with coefficients that   satisfy estimate \eqref{EQ:bound} in a neighborhood of $t=0$.
   Then \eqref{EQ-ito} holds.
\end{theorem}
\subsection{Proof of Theorem \ref{Thm2}}
By linearity, it is enough to  consider a single term in  \eqref{EQ:q-Int-2}. We first consider the constant term, $f(x,t)=b(t)$. Then
$\nabla_x f=0$ as the expression $f(x,s)-f(y,s)$ under the integral in the definition \eqref{EQ:nabla} vanishes.
Similarly, $\Delta_x f=0$, as the expression $(z-x)f(y,s)+(x-y)f(z,s)+(y-z)f(x,s) $ in the definition \eqref{EQ:Delta} vanishes.
So \eqref{EQ-ito} in this case reduces to the $q$-integral identity
$b(t)-b(0)=\int_0^t (\calD_{q,s} b)(s)d_qs$ which was already recalled in \eqref{EQ:antideriv}.

Now we consider a non-constant term of degree $m+1$. We write \eqref{EQ:q-Int-2} as
\begin{multline}
  \label{EQ-one-term-1}
  b(t)h_{m+1}(\Wq_t;t)=[m+1]_q\int_0^t b(s) h_m(\Wq_s;s)\dd \Wq \\ +\int_0^t h_{m+1}(\Wq_{qs};qs)(\calD_{q,s} b)(s)d_qs.
\end{multline}
We first note the following identity.
\begin{lemma}  \label{L:BW} For a fixed $s>0$, we have
$\nabla_x^{(s)} \left(h_{m+1}(x;s)\right)=[m+1]h_m(x;s)$.
\end{lemma}
\begin{proof}
  This is a special case of \cite[Lemma 2.3]{Bryc-Wesolowski-2013-gener}, applied to $\tau=\theta=0$.
\end{proof}
Thus, noting that $\nabla_x^{(s)}$ acts only on variable $x$, we have
 \begin{multline}
  \label{EQ-one-term-2}
  b(t)h_{m+1}(\Wq_t;t) \\= \int_0^t \nabla_x^{(s)} \left(b(s) h_{m+1}(x;s)\right)\Big|_{x=\Wq_s}\dd \Wq+\int_0^t h_{m+1}(\Wq_{qs};qs)(\calD_{q,s} b)(s)d_qs.
\end{multline}
Next, we consider the second term on the right hand side of \eqref{EQ-one-term-1}. Using the $q$-product identity
\begin{equation}
  \label{EQ-q-prod}
  \calD_{q,s}(b(s)h(s))=b(s)\calD_{q,s}(h(s))+h(qs)\calD_{q,s}(b(s)),
\end{equation}
we write
$$h_{m+1}(x;qs)(\calD_{q,s} b)(s)=\calD_{q,s}\left(b(s)h_{m+1}(x,s)\right) - b(s)\calD_{q,s}(h_{m+1}(x,s)).$$
 This recovers the second term in \eqref{EQ-ito}: \eqref{EQ-one-term-2} becomes
 \begin{multline}
  \label{EQ-one-term-3}
  b(t)h_{m+1}(\Wq_t;t)= \int_0^t \nabla_x^{(s)} \left(b(s) h_{m+1}(x;s)\right)\Big|_{x=\Wq_s}\dd \Wq \\ +\int_0^t
  \calD_{q,s}\left(b(s) h_{m+1}(x;s)\right)\Big|_{x=\Wq_{qs}}d_qs  +\int_0^t D_x^{(s)}\left(b(s)h_{m+1}(x;s)\right) \Big|_{x=\Wq_{qs}}d_qs,
\end{multline}
where $D_x^{(s)}$ is a linear operator on polynomials in variable $x$, whose action on monomials $x^n=\sum_{j=0}^n a_j(s)h_j(x;s)$ is defined by
$$
D_x^{(s)}(x^n)= - \sum_{j=0}^n a_j(s)\calD_{q,s}\left(h_j(x;s)\right).
$$
It remains to identify $D_x^{(s)}$ with $\Delta_x^{(s)}$. To do so we use the approach from \cite{Bryc-Wesolowski-2013-gener}. We fix $s>0$ and introduce an auxiliary linear operator $A$ that acts on polynomials $p(x)$  in variable $x$   by the formula
\begin{equation}
  \label{D2P}
  A(p(x))=D_x^{(s)}\left(x p(x)\right)-x D_x^{(s)}\left(p(x)\right).
\end{equation}
%Similarly to \cite[Lemma 2.2]{Bryc-Wesolowski-2013-gener} we have:
\begin{lemma}\label{L:B}
\begin{equation}
  \label{EQ:B2} A\left(h_{m}(x;s)\right)=[m]_qh_{m-1}(x;qs).
\end{equation}

\end{lemma}
\begin{proof} By definition, $A\left(h_{m}(x;s)\right)=D_x^{(s)}\left(x h_{m}(x;s)\right)-x D_x^{(s)}\left(h_{m}(x;s) \right)$, and we evaluate each term separately.
From recursion \eqref{EQ:rec-q-hermite} we get
\begin{equation}\label{EQ:Dx}
D_x^{(s)}\left(x h_{m}(x;s)\right)
=-\calD_{q,s}\left(h_{m+1}(x;s)\right)-s[m]_q\calD_{q,s}\left(h_{m-1}(x;s)\right).
\end{equation}
Next, we have
\begin{multline*}
   x D_x^{(s)}\left(h_{m}(x;s)\right)=-x \calD_{q,s}\left(h_{m}(x;s)\right)=- \calD_{q,s}\left(xh_{m}(x;s)\right) \\ =
- \calD_{q,s}\left(h_{m+1}(x;s)+s [m] h_{m-1}(x;s) \right)\\=
- \calD_{q,s}\left(h_{m+1}(x;s)\right)- [m]_q \calD_{q,s}\left(s  h_{m-1}(x;s) \right).
\end{multline*}
Since $\calD_{q,s}(s)=1$, applying \eqref{EQ-q-prod} we get
$$xD_x^{(s)}\left(h_{m}(x;s)\right)
=- \calD_{q,s}\left(h_{m+1}(x;s)\right) -s[m]_q \calD_{q,s}\left(h_{m-1}(x;s)\right) - [m]_q h_{m-1}(x;qs).$$
To end the proof, we subtract this from \eqref{EQ:Dx}.
\end{proof}
Using martingale property \eqref{EQ:mart} and Lemma \ref{L:BW}, we rewrite \eqref{EQ:B2} as
\begin{multline*}
A\left(h_{m}(x;s)\right)=\int [m]_q h_{m-1}(y;s) P_{qs,s}(x,dy)\\ =\iint \frac{h_{m}(z;s)-h_m(y;s)}{z-y}\nu_{s,y}(dz)P_{qs,s}(x,dy).
\end{multline*}
So by linearity, for any polynomial $p$ in variable $x$ we have
$$
(A p)(x)=\iint \frac{p(z)-p(y)}{z-y}\mu_{x,s}(dy,dz).
$$
Since \eqref{D2P} gives $D_x^{(s)}(x^n)=A(x^{n-1})+x D_x^{(s)}(x^{n-1})$ and $D_x^{(s)}(x)=-\calD_{q,s}(h_1(x;s))=0$ this determines $D_x^{(s)}$ on monomials:
$$
D_x^{(s)}(x^n)=\sum_{k=0}^{n-1}x^k A(x^{n-k-1})=\iint \left(\sum_{k=0}^{n-1} x^k\frac{z^{n-k-1}-y^{n-k-1}}{z-y}\right)\mu_{x,s}(dy,dz).
$$

 By  the geometric sum formula, the integrand is
$$\sum_{k=0}^{n-1} x^k\frac{z^{n-k-1}-y^{n-k-1}}{z-y}=\frac{ (x-z)y^n+(y-x)z^n + (z-y)x^n}{(z-x)(x-y)  (y-z)}.$$
This shows that $ D_x^{(s)}=\Delta_x^{(s)}$ on polynomials. With this identification, \eqref{EQ-one-term-3}  concludes the proof of Theorem \ref{Thm2}.

\begin{remark}\label{Rem-Ito} Operators $\nabla_x^{(s)}$ and $\Delta_x^{(s)}$ are well defined on functions with bounded second derivatives in variable $x$, so we expect that   \eqref{EQ-ito} holds in more generality.

For an analytic function $f(x)$ one can use contour integration to verify that  each term in  \eqref{EQ-ito}  can be approximated uniformly on compacts by the same expression applied to a polynomial.
% Since an analytic function  from Proposition \ref{P4.5} can be approximated uniformly on compacts by polynomials,  from Lemma  \ref{L:BW}  we deduce that the
%expansion \eqref{EQ-L2}  of   $\nabla_{x}^{(s)}f$ is just
%the expansion of $f$ with shifted coefficients.
Thus formula   \eqref{EQ-ito} for polynomials yields
 $$
 f(\Wq_t)=f(0)+\int_0^t (\nabla_x^{(s)} f)(\Wq_s,s)\dd\Wq+\int_0^t   (\Delta_x^{(s)} f)(\Wq_{qs},s) d_q s.
 $$

\end{remark}

\arxiv{
\begin{proof}[Proof of Remark \ref{Rem-Ito}]
Using Cauchy formula with $R$ large enough, we have
\begin{equation}\label{CF1}
g(x,t):=(\nabla_x^{(t)}f)(x,t)=\frac{1}{2\pi i} \oint_{|\zeta|=R} \frac{f(\zeta)}{(\zeta-x)(\zeta-y)}d\zeta d\nu_{x,t}(dy).
\end{equation}
With $p_n(x)=\sum_{k=0}^n a_k x^k\to f(x)$, we see that
\begin{multline*}
g(x,t)=\frac{1}{2\pi i} \oint_{|\zeta|=R} \frac{p_n(\zeta)}{(\zeta-x)(\zeta-y)}d\zeta d\nu_{x,t}(dy)\\+\frac{1}{2\pi i} \oint_{|\zeta|=R} \frac{f(\zeta)-p_n(\zeta)}{(\zeta-x)(\zeta-y)}d\zeta d\nu_{x,t}(dy)
\end{multline*}
can be approximated uniformly in $x$ from  the support of $\gamma_{t;q}$ by polynomials. In particular, expanding $p_n(x)$ into the $q$-Hermite polynomials, from Lemma  \ref{L:BW}  we deduce that
$$
g(x,t)=\sum_{k=0}^\infty \frac{b_n(t)}{[n]!}\nabla_x^{(t)}(h_n)(x,t) =
\sum_{k=0}^\infty \frac{b_{n+1}(t)}{[n]!} h_n(x;t)
$$
and the convergence is uniform over $x$ from the support of $\gamma_t$.
The estimates from proof of Proposition \ref{P4.5} show that  $\int_0^t (\nabla_x^{(s)} f)(\Wq_s,s)\dd\Wq$ is well defined.

Next we note that
$$(\Delta_x^{(s)}f)(x,t)=\frac{1}{2\pi i} \oint_{|\zeta|=R} \frac{f(\zeta)}{(\zeta-x)(\zeta-y)(\zeta-z)}d\zeta d\mu_{x,s}(dy,dz),$$
so $(\Delta_x^{(s)}f)(x,t)$ is also a limit of  $(\Delta_x^{(s)}p_n)(x,t)$, and the convergence is uniform over $x$ from the support of $\gamma_t$ and over all $ s\in[0,t]$.

By  Theorem \ref{Thm2}, formula \eqref{EQ-ito} is valid for polynomials $p_n(x)$ and the contour integral identities show that we can pass to the limit  as $n\to\infty$.

\end{proof}
}
\section{A $q$-analogue of the stochastic exponential} \label{Sec-SDE}

In this section we consider a more general function $f(x,t)$ for which the mean-square limit \eqref{EQ:l.i.m} is applicable.
Our goal is to exhibit a solution to  the ``differential equation" $\dd Z=a Z \dd\Wq$ which we interpret in integral form as
\begin{equation}
  \label{ODE}
  Z_t=c+ a \int_0^t Z_s\dd \Wq
\end{equation}
with deterministic parameters $a,c\in\RR$. In view of the fact that we can  integrate only instantaneous functions of $\Wq_t$,   we seek a solution in this form, and we seek the series expansion.
The solution parallels the development in Ito theory and relies on the identity \eqref{WdW}
%$$\int_0^t h_n(\Wq_s,s)\dd\Wq=\frac{1}{[n+1]}h_{n+1}(\Wq_t;t)$$
which is a special case of our definition \eqref{EQ:q-Int-2}.
Thus,   as a mean-square expansion the solution of \eqref{ODE} is
\begin{equation}
  \label{EQ:Z-series}
  Z_t=c \sum_{n=0}^\infty \frac{a^n h_n(\Wq_t;t)}{[n]_q!}=c\prod_{k=0}^\infty \left(1-(1-q)a q^k \Wq_t +a^2t q^{2k}\right)^{-1}.
\end{equation}
Formal calculations give
$$c+a\int_0^t Z_s \dd\Wq= c+c a \sum_{n=0}^\infty \frac{a^n h_{n+1}(\Wq_t;t)}{[n+1]_q!}=Z_t,$$
as $h_0(x;t)=1$.

Note that the product expression on the right hand side of \eqref{EQ:Z-series} is well defined for all $t$, as with probability one $|\Wq_t|< 2\sqrt{t}/\sqrt{1-q}$.
 However,   the solution  of \eqref{ODE} ``lives" only on the finite interval $0\leq t < a^{-2}(1-q)^{-1}$.
 This is because for $0<q<1$ the product $\prod_{j=1}^{\infty}(1-q^j)$ converges, so the  series
$$
\sum_n \frac{a^{2n}}{[n]!} \int_0^t s^n d_qs=\sum_n \frac{a^{2n}t^{n+1}}{[n+1]!}=\sum_n \frac{(1-q)^{n}a^{2n}t^{n+1}}{\prod_{j=2}^{n+1}(1-q^j)}
$$
converges when $(1-q)a^2t<1$ and hence Corollary \ref{Cor} can be applied only to this range of $t$. This is the same range of $t$ where $(Z_t)$ is a martingale \cite[Corollary 4]{szabl2012q}.

\arxiv{
\section{Deterministic integrands}\label{Sec-Determ} In the Ito-Wiener theory, the distribution of the integral of a deterministic function $b(s)$  is Gaussian with the variance $\int_0^t b^2(s)ds$.
In this section we show that for $0<q<1$ the distribution of the random variable $\int_0^1 b(s)\dd \Wq$ depends on $b(s)$ and $q$ in a more complicated way.
This gives rise to a plethora of  $q$-Gaussian laws that all converge to the Gaussian law as $q\to 1$.
To illustrate this phenomenon, we consider
$$Z=\int_0^1 s^r \dd \Wq=\sum_{k=0}^\infty q^{kr}\left(\Wq_{q^k}-\Wq_{q^{k+1}}\right)
$$ for $r\geq 0$. Then $E(Z)=0$ and from \eqref{EQ-iso} we see that
\begin{equation}
  \label{EZ^2}
  E(Z^2)=\int_0^1 s^{2r}d_qs=\frac{1-q}{1-q^{2r+1}}=\frac{1}{[2r+1]_q}.
\end{equation}
To show that the distribution of $Z$ depends on $r$ and $q$ in a complicated way, we compute the fourth moment. Recall that the $4$-th moment of $\gamma_{t;q}$ is $t^2(2+q)$, where $t$ is the variance. We will show that this pattern fails for $Z$.
}

\arxiv{
Noting that by martingale property $E(\Wq_{u}-\Wq_t|\Wq_s)=0$ for $s<t<u$,  to compute $E(Z^4)$ we need only the formulas
\begin{equation}
  \label{EQ:DW^4}
  E\left((\Wq_{t}-\Wq_{s})^4\right)=(t-s) ((q+2) t-3 q s)
\end{equation}
for $t<s$ and
\begin{equation}
  \label{EQ:DW2DW2}
E\left((\Wq_{t_2}-\Wq_{t_1})^2(\Wq_{u_2}-\Wq_{u_1})^2\right)=(u_2-u_1)(t_2-t_1)
\end{equation}
\begin{equation}
  \label{EQ:DW1DW3}
E\left((\Wq_{t_2}-\Wq_{t_1})(\Wq_{u_2}-\Wq_{u_1})^3\right)=-(1-q)(u_2-u_1)(t_2-t_1)
\end{equation}
for $t_1<t_2\leq u_1<u_2$.

%\longcomment{
}

\arxiv{
The calculations use  the first four martingale polynomials: the first three were already listed in Example  \ref{Ex:4.2}  and the fourth one is
$h_4(x;t)=x^4-\left(q^2+2
   q+3\right) t x^2 +\left(q^2+q+1\right) t^2$.
We use them to compute recurrently conditional moments. This gives martingale property  that we already used, the quadratic martingale property $E((\Wq_t)^2|\Wq_s)=(\Wq_s)^2+t-s$, and
a bit more cumbersome formulas
\begin{equation}
  \label{m[3]}
E((\Wq_t)^3|\Wq_s)=(\Wq_s)^3+(t-s)(2+q)\Wq_s
\end{equation}

\begin{equation}
  \label{m[4]}
E((\Wq_t)^4|\Wq_s)=(\Wq_s)^4+(t-s)(3+2q+q^2)(\Wq_s)^2+(t-s)((2+q)t-(1+q+q^2)s)
\end{equation}
After a calculation these formulas  give
$$E\big((\Wq_t-\Wq_s)^4\big|\Wq_s\big)= (t-s)\left((1-q)^2  (\Wq_s)^2+(q+2) t- \left(q^2+q+1\right) s\right).$$ To derive \eqref{EQ:DW^4} we  take the expected value of this expression.

%To derive \eqref{EQ:DW^4} we use the first four martingale polynomials
%$$\{h_n(x,t):n=0,\dots,4\}=\left\{1,x,x^2-t,x^3-(q+2) t x,\left(q^2+q+1\right) t^2-\left(q^2+2
%   q+3\right) t x^2+x^4\right\}$$
%   to compute recurrently conditional moments $E((\Wq_t)^n|\Wq_s)$,  $n=0,\dots,4$. Then we compute $$E\big((\Wq_t-\Wq_s)^4\big|\Wq_s\big)= (t-s)\left((1-q)^2  (\Wq_s)^2+(q+2) t- \left(q^2+q+1\right) s\right)$$
%and  take the expected value of this expression.
%}

%To derive \eqref{EQ:DW^4} we use the first four martingale polynomials
% $h_n(\Wq_t;t)$ \comment{List first four}
%to compute
%$$E\big((\Wq_t-\Wq_s)^4\big|\Wq_s\big)= (t-s)\left((1-q)^2  (\Wq_s)^2+(q+2) t- \left(q^2+q+1\right) s\right)$$
%and then take the expected value of this expression.

To derive \eqref{EQ:DW2DW2}, we use martingale property and quadratic martingale property to compute
$$E\big((\Wq_{u_2}-\Wq_{u_1})^2\big|\Wq_{t_2}\big)= E\big((\Wq_{u_2}-\Wq_{t_2})^2\big|\Wq_{t_2}\big)-E\big((\Wq_{u_1}-\Wq_{t_2})^2\big|\Wq_{t_2}\big)=u_2-u_1.$$

To derive \eqref{EQ:DW1DW3} we use again the  formulas for conditional moments to compute $E\big((\Wq_{u_2}-\Wq_{u_1})^3\big|\Wq_{u_1}\big)=-(1 - q) (u_2 - u_1) \Wq_{u_1}$. This gives
$$E\left((\Wq_{t_2}-\Wq_{t_1})(\Wq_{u_2}-\Wq_{u_1})^3\right)=-(1-q)(u_2-u_1)\left(E(\Wq_{t_2}\Wq_{u_1})- E(\Wq_{t_1}\Wq_{u_1})\right).$$
}

\arxiv{We now compute  $E(Z^4)$. Denoting $\xi_k=\Wq_{q^{k}}-\Wq_{q^{k+1}}$, we see that
$$E(Z^4)= \sum_{k=0}^\infty q^{4rk} E\xi_k^4+ 6 \sum_{k=1}^\infty\sum_{j=0}^{k-1} q^{2r(j+k)}E(\xi_j^2\xi_k^2)+
%%%4  \sum_{k=1}^\infty\sum_{j=0}^{k-1} q^{r(j+3 k)}E(\xi_j\xi_k^3)
4  \sum_{k=1}^\infty\sum_{j=0}^{k-1} q^{r(3j+ k)}E(\xi_k\xi_j^3),
$$
as all other terms vanish.
From \eqref{EQ:DW^4} we see that $E((\Wq_{t q^{k}}-\Wq_{t q^{k+1}})^4)=(1-q)^2 t^2 q^{2 k}(2+3q)$, so
$$
\sum_{k=0}^\infty q^{4rk} E\xi_k^4=\frac{(1-q)^2 (2+3 q)}{1-q^{4 r+2}}.
$$
 From \eqref{EQ:DW2DW2} we get
 $$
 \sum_{k=1}^\infty\sum_{j=0}^{k-1} q^{2r(j+k)}E(\xi_j^2\xi_k^2)=(1-q)^2\sum_{k=1}^\infty\sum_{j=0}^{k-1} q^{(2r+1)(j+k)} =\frac{(1-q)^2 q^{2 r+1}}{\left(1-q^{2 r+1}\right)^2
   \left(1+q^{2 r+1}\right)}.
 $$
 From \eqref{EQ:DW1DW3} we get
 $$
 \sum_{k=1}^\infty\sum_{j=0}^{k-1} q^{r(3j+ k)}E(\xi_k\xi_j^3)= -(1-q)^3 \sum_{j=0}^{k-1} q^{(3r+1)j+(r+1)k}  =-\frac{(1-q)^3 q^{r+1}}{\left(1-q^{r+1}\right) \left(1-q^{4
   r+2}\right)}.
 $$
 This gives
% $$E(Z^4)=\frac{(1-q)^2 \left(6 q^{r+1}-q^{r+2}-4 q^{2 r+1}+3 q^{2 r+2}+q^{3
%   r+3}-3 q-2\right)}{\left(q^{r+1}-1\right) \left(1-q^{2
%   r+1}\right)^2 \left(q^{2 r+1}+1\right)}$$
    $$E(Z^4)=\frac{(1-q)^2 \left(  2+3 q   -6 q^{r+1}+q^{r+2}+4 q^{2 r+1}-3 q^{2 r+2}-q^{3
   r+3}\right)}{\left(1-q^{r+1}\right) \left(1-q^{2
   r+1}\right)^2 \left(1+q^{2 r+1}\right)}.$$
So using \eqref{EZ^2} we get
  %$$\frac{E(Z^4)}{(E(Z^2))^2}=\frac{6 q^{r+1}-q^{r+2}-4 q^{2 r+1}+3 q^{2 r+2}+q^{3 r+3}-3
%   q-2}{\left(q^{r+1}-1\right) \left(q^{2 r+1}+1\right)}$$
    $$\frac{E(Z^4)}{(E(Z^2))^2}=\frac{2+3q-6 q^{r+1}+q^{r+2}+4 q^{2 r+1}-3 q^{2 r+2}-q^{3 r+3} }{\left(1-q^{r+1}\right) \left(1+q^{2 r+1}\right)}.$$
    }

\subsection*{Acknowledgements}
I would like to thank Magda Peligrad for positive feedback and encouragement.  Referee's comments helped to improve the presentation.
   \bibliographystyle{plain}
\bibliography{q-int,../vita}

\end{document}